\documentclass[11pt]{article}
 
 \usepackage{amsthm}
 \usepackage{enumitem}
 \usepackage{hyperref}
 \usepackage{amssymb}
 \usepackage{amsmath}
 \usepackage{dsfont}
 \usepackage{mathtools}
 \usepackage{xcolor}
 \usepackage{mathrsfs}
 \usepackage{amsmath}
\usepackage[uline]{hhtensor}
 \usepackage{float}
\usepackage{graphicx,graphics,subfigure}
\usepackage{authblk}
\usepackage[titlenumbered,ruled]{algorithm2e}
 \usepackage{comment}

\usepackage{booktabs}

 \usepackage[left=3cm,right=3cm,top=3cm,bottom=3cm]{geometry}

\newtheorem{theorem}{Theorem}
\theoremstyle{definition}
\newtheorem{definition}{Definition}[section]

 \newtheorem{remarkex}{Remark}
\newenvironment{remark}
  {\remarkex}
  {\endremarkex}

\theoremstyle{remark}

\theoremstyle{definition}

\theoremstyle{remark}

\theoremstyle{remark}

\title{A new approximation method for solving stochastic differential equations}

\author{Faezeh Nassajian Mojarrad}
\affil{Max Planck Institute for Informatics, Germany}

\date{}

\begin{document}
\maketitle

\begin{abstract}
We present a novel solution method for Itô stochastic differential equations (SDEs).
We subdivide the time interval into sub-intervals, then we use the quadratic polynomials for the approximation between two successive intervals. 
The main properties of the stochastic numerical methods, e.g. convergence, consistency, and stability are analyzed.
We test the proposed method in SDE problem, demonstrating promising results.
\end{abstract}

\section{Introduction}
Mathematical modeling of many real-world phenomena involving random noisy perturbations cannot be achieved using ordinary differential equations (ODEs). Therefore, they are often modeled using stochastic differential equations (SDEs) to make the model more realistic \cite{Kloeden1995,Oksendal2013}.
Some applications of SDEs include theoretical physics, molecular biology, population dynamics, and mathematical finance \cite{Kloeden2012}.
SDEs have been studied from various aspects in recent
decades, see e.g. \cite{Wang2015,Zhang2017,Nouri2020,Jornet2020,Calatayud2020,Shahmoradi2021,Hu2021,Cortés2022}.

Models based on  SDEs  are significantly more computationally challenging than deterministic ones.
 Due to the high computational cost for numerical simulations of stochastic models in various applications \cite{Atalla1986,Xiu2002}, there have been many attempts to find efficient numerical tools.
 The most effective and widely used method for solving SDEs is the simulation of sample paths using time discrete approximations on digital computers \cite{Milstein1995}.
  This technique involves discretizing the time interval  into finite steps and iteratively generating approximate values for the solution process at these points. The resulting simulated sample paths can be evaluated with standard statistical methods to assess the quality of the approximation and its closeness to the exact solution.
  Among numerical schemes with the same convergence order, the most economical methods possess the largest mean-square  stability regions and the lowest computational costs. Explicit numerical schemes equipped with ODE solvers, are a successful example of this type, e.g. 
explicit stochastic orthogonal Runge–Kutta Chebyshev schemes, which
  were proposed by \cite{Abdulle2008,Li2008}. In \cite{Abdulle2018,Abdulle2013,Komori2012,Komori2013} these schemes have been developed.
  A
class of exponential integrators for SDEs, namely, local linearization approaches have been proposed in \cite{Jimenez1999,Jimenez2002}.
In \cite{Yin2015}
has been presented a novel explicit Euler–Maruyama (EM) scheme using deterministic error correction terms.
 An explicit Steklov mean EM scheme  for the approximation
of SDEs was constructed in \cite{Díaz2016},
and was developed in \cite{Díaz2017}, where the implicit
split-step linear Steklov mean scheme was proposed for SDEs with locally Lipschitz coefficients.

Let us consider the following Itô SDE
\cite{Gard1988,Mao1997}:
\begin{equation}
\label{eq:main}
  dX(t)=f(t,X(t)) \;dt+ g(t,X(t)) \;dW(t), ~ 0\leq t\leq T,
\end{equation}
with the initial condition
\begin{equation}
   X(0)=X_0.  
\end{equation}
Note that one has to consider 
 a  probability space with an increasing family of $\sigma$-fields of 
$(\Omega,\mathcal{F},P;\mathcal{F}_t)$, and a family of stochastic processes
$\{ X(t),W(t) \}$
defined on it such that \cite{Kloeden2012}
\begin{itemize}
    \item $X(t)$ and $W(t)$ are continuous in $t$ with probability one and $W(0)=0$. 
    \item $X(t)$ and $W(t)$ are $\mathcal{F}_t$-measurable.
\item $W(t)$ is a $\mathcal{F}_t$-martingale such that
\begin{equation}
 \mathbb{E}((W(t)-W(s))^2|\mathcal{F}_s)=t-s, ~t\geq s,  
\end{equation}
\item $X(t)$ and $W(t)$ satisfy the following relation
\begin{equation}
  X(t)=  X(0)+\int_0 ^t f(s,X(s))\; ds+
  \int_0 ^t g(s,X(s))\; dW(s).
\end{equation}
\end{itemize}
Classic stochastic schemes such as Euler–Maruyama (EM) and Milstein schemes are as follows:
 EM  method
 \cite{Maruyama1955}
 \begin{equation}
  \tilde{X}_{i+1}=\tilde{X}_i+f(t_{i}, \tilde{X}_{i}) \Delta t_i+ 
  g(t_i, \tilde{X}_i) \Delta W_i,
 \end{equation}
where $\tilde{X}_i$
is an approximation of the solution to \eqref{eq:main},
$\Delta t_i=t_{i+1}-t_i$ and $\Delta W_i=W_{t_{i+1}}-W_{t_i}$, and Milstein method 
\cite{Milstein1995}
 \begin{equation}
  \tilde{X}_{i+1}=\tilde{X}_i+f(t_i, \tilde{X}_i) \Delta t_i+ 
  g(t_i, \tilde{X}_i) \Delta W_i-\frac{1}{2}
  g(t_i, \tilde{X}_i)\frac{\partial g(t_i,\tilde{X}_i)}{\partial x}((\Delta  W_i)^2-\Delta t_i).
 \end{equation}
The outline of the manuscript is as follows: 
in Section  \ref{sec:methodology}, we  describe how to  use quadratic approximating functions to develop the algorithm. Section
\ref{sec:convergence} 
contains the analysis of the stability, consistency and convergence of the proposed method.
In Section \ref{sec:numerical} we test the validity and eﬀectiveness of the proposed strategy.
Finally, Section \ref{sec:conclusion} contains a discussion about the obtained
results, as well as future work.

\section{Methodology}
\label{sec:methodology}
In this work, we consider  the  equation for geometric Brownian motion:
\begin{equation}
\label{eq:sde}
  dX(t)=\mu X(t) \;dt+\sigma X(t) \;dW(t). 
\end{equation}
where $\mu$ is the percentage drift and 
$\sigma$
is the percentage volatility. 

In order to integrate the system, we subdivide the interval $[0,T]$ into sub-intervals obtained from a partition
$$0=t_0<t_1<\cdots< t_i <\cdots< t_N=T.$$
Note that $N$ is an even number.
In all our examples, this partition will be regular: $t_i=i\Delta t$. We assume that the approximate numerical values for $X(t)$ have been determined  at the grid points $t_i$, $i=0,1,, \cdots, 2n$, for $t_i<T$. Our goal is to obtain the values of the function $X(t)$ at the next two time steps $t_{2n+1}$ and $t_{2n+2}$. 

We now explain the method in more detail. By integrating\eqref{eq:sde}  on  $[t_{2n},t_{2n+2}]$, we have
\begin{equation}
\label{eq:integral1}
\tilde{X}_{2n+2}- \tilde{X}_{2n}= \mu \int_{t_{2n}}^{t_{2n+2}} X(t)\; dt+ \sigma \tilde{X}_{2n}
(E^2-1)W_{2n}.
\end{equation}
where 
\begin{equation*}
(E^2-1)W_{2n}=W((2n+2)\Delta t)-W((2n)\Delta t).    
\end{equation*}
We can approximate $X(t)$ in the interval $[t_{2n},t_{2n+2}]$ using a quadratic interpolation functions \cite{Kumar2006}. We get the following approximation 
\begin{equation}
\label{eq:interp1}
  X(t)= \sum_{j=2n}^{2n+2} \psi_{n,j}(t)\tilde{X}_j,
\end{equation}
where $\psi_{n,j}(t)$ are the  $j$-th quadratic interpolating functions at the  step $n+1$.
Functions $\psi_{n,j}(t)$, $j= 2n, 2n+1, 2n+2$
are defined by
\begin{align}
\begin{split}
\label{eq:psi}
 \psi_{n,2n}(t)&=\frac{(t-(2n+1)\Delta t)(t-(2n+2)\Delta t)}{2\Delta t^2},\\
 \psi_{n,2n+1}(t)&=\frac{(t-(2n)\Delta t)(t-(2n+2)\Delta t)}{-\Delta t^2},\\
 \psi_{n,2n+2}(t)&=\frac{(t-(2n)\Delta t)(t-(2n+1)\Delta t)}{2\Delta t^2}.
 \end{split}
\end{align}
In order to determine the values of $X(t)$ at $t_{2n+1}$ and $t_{2n+2}$, we substitute \eqref{eq:interp1} into \eqref{eq:integral1}. We get
\begin{equation}
\label{eq:m1}
\tilde{X}_{2n+2}- \tilde{X}_{2n}= \mu \int_{t_{2n}}^{t_{2n+2}} (\sum_{j=2n}^{2n+2} \psi_{n,j}(t)\tilde{X}_j)\; dt+ \sigma \tilde{X}_{2n}
(E^2-1)W_{2n}.
\end{equation}
Hence, we have
\begin{equation}
 \label{eq:res1}
 \tilde{X}_{2n+2}- \tilde{X}_{2n}=\mu (\frac{\Delta t}{3} \tilde{X}_{2n}+
 \frac{4\Delta t}{3} \tilde{X}_{2n+1}+
 \frac{\Delta t}{3} \tilde{X}_{2n+2})
 + \sigma \tilde{X}_{2n}
 (E^2-1)W_{2n}.
\end{equation}
We observe that \eqref{eq:res1} requires the values of $X(t)$ at $t_{2n+1}$ and $t_{2n+2}$.

We
use the same technique for the interval $[t_{2n},t_{2n+1}]$.
By integrating \eqref{eq:sde}  on  $[t_{2n},t_{2n+1}]$, we get the following formula
\begin{equation}
\label{eq:integral2}
\tilde{X}_{2n+1}- \tilde{X}_{2n}=\mu \int_{t_{2n}}^{t_{2n+1}} X(t)\; dt+ \sigma \tilde{X}_{2n}
\Delta W_{2n}.
\end{equation}
Similarly, we  approximate $X(t)$ in the interval $[t_{2n},t_{2n+1}]$ using a quadratic interpolation functions. 
\begin{equation}
\label{eq:interp2}
  X(t)= \sum_{j=2n,2n+\frac{1}{2},2n+1} \psi'_{n,j}(t)\tilde{X}_j,
\end{equation}
where $j$-th quadratic interpolating functions $\psi'_{n,j}(t)$, $j= 2n, 2n+\frac{1}{2}, 2n+1$ at the  step $n+1$
are defined by
\begin{align}
\begin{split}
\label{eq:psi'}
 \psi_{n,2n}(t)&=\frac{(t-(2n+\frac{1}{2})\Delta t)(t-(2n+1)\Delta t)}{\frac{1}{2}\Delta t^2},\\
 \psi_{n,2n+\frac{1}{2}}(t)&=\frac{(t-(2n)\Delta t)(t-(2n+1)\Delta t)}{-\frac{1}{4}\Delta t^2},\\
 \psi_{n,2n+1}(t)&=\frac{(t-(2n)\Delta t)(t-(2n+\frac{1}{2})\Delta t)}{\frac{1}{2}\Delta t^2}.
 \end{split}
\end{align}
By substituting \eqref{eq:interp2} into 
\eqref{eq:integral2}, we have
\begin{equation}
 \label{eq:res2}
 \tilde{X}_{2n+1}- \tilde{X}_{2n}=\mu (\frac{\Delta t}{6} \tilde{X}_{2n}+
 \frac{2\Delta t}{3} \tilde{X}_{2n+\frac{1}{2}}+
 \frac{\Delta t}{6} \tilde{X}_{2n+1})
 + \sigma \tilde{X}_{2n}
 \Delta W_{2n}.
\end{equation}
By using \eqref{eq:interp1}, the value of $X(t)$ at $t=t_{2n+\frac{1}{2}}$ can be calculated as follows
\begin{equation}
  \label{eq:mid_approx}  
\tilde{X}_{2n+\frac{1}{2}}=\frac{3}{8}  \tilde{X}_{2n}+\frac{3}{4}\tilde{X}_{2n+1}-\frac{1}{8}  \tilde{X}_{2n+2}.
\end{equation}
We substitute \eqref{eq:mid_approx}  into \eqref{eq:res2}, we obtain
\begin{equation}
 \label{eq:res3}
 \tilde{X}_{2n+1}- \tilde{X}_{2n}=\mu (\frac{5\Delta t}{12} \tilde{X}_{2n}+
 \frac{2\Delta t}{3} \tilde{X}_{2n+1}-
 \frac{\Delta t}{12} \tilde{X}_{2n+2})
 + \sigma \tilde{X}_{2n}
 \Delta W_{2n}.
\end{equation}
Since $\tilde{X}_{2n}$ is known, we can solve 
\eqref{eq:res1} and \eqref{eq:res3} to obtain the unknowns $\tilde{X}_{2n+1}$ and $\tilde{X}_{2n+2}$. Therefore, we get
\begin{subequations}
\label{eq:method}
\begin{equation}
\label{eq:method1}
 \tilde{X}_{2n+1}=  \alpha_n \tilde{X}_{2n},
 \end{equation} 
 \begin{equation}
 \label{eq:method2}
 \tilde{X}_{2n+2}=\beta_n \tilde{X}_{2n},
\end{equation}    
\end{subequations}
for $n=0, 1,\cdots, \frac{N-2}{2}$. $\alpha_n$ and $\beta_n$ are defined by 
\begin{equation*}
\alpha_n=\frac{1-\frac{(\mu \Delta t)^2}{6}
-\sigma \frac{\mu \Delta t}{12}
(E^2-1)W_{2n}
 +\sigma (1-\frac{\mu \Delta t}{3})\Delta W_{2n}}{1-
 \mu \Delta t+
 \frac{(\mu \Delta t)^2}{3}}, 
\end{equation*}
and
\begin{equation*}
 \beta_n=\frac{1+\frac{\mu \Delta t}{3}+\frac{4\mu \Delta t}{3}\alpha_n+\sigma (E^2-1)W_{2n}}
 {1-\frac{\mu \Delta t}{3}}.   
\end{equation*}
We continue this process to progress the integration until we achieve the final time $T$.

\section{Stability, consistency and convergence analysis of the  method}
\label{sec:convergence} 

\begin{definition}
The sequence $\{\tilde{X}_{n}\}_{n=0}^{\infty}$ is mean-square stable if
\begin{equation}
   \lim_{n \to \infty} \mathbb{E} (|\tilde{X}_{n}|^2)=0.
\end{equation}
\end{definition}

\begin{theorem}
\label{theorem:stability}
Under certain conditions on the parameters, 
The proposed method
\eqref{eq:method}
is stable 
in mean square 
for approximating the solution of stochastic diffusion equation \eqref{eq:sde}.
\end{theorem}

\begin{proof}
To analyze mean-square stability of our method, we see from \eqref{eq:method2} that 
\begin{align}
\label{eq:stability}
 \begin{split}
 |\tilde{X}_{2n+2}|^2 &= \frac{1}{
 |1-
 \mu \Delta t+
 \frac{(\mu \Delta t)^2}{3}|^2|1-\frac{\mu \Delta t}{3}|^2} 
 \bigg(
|1+\frac{2\mu \Delta t}{3}-\frac{(\mu \Delta t)^3}{9}|^2
\\&+
 |\sigma (1-\mu \Delta t+\frac{2(\mu \Delta t)^2}{9})
(E^2-1)W_{2n}|^2\\
&+
|\sigma \frac{4\mu \Delta t}{3}
(1-\frac{\mu \Delta t}{3})\Delta W_{2n}|^2\\
&+
|2\sigma (1+\frac{2\mu \Delta t}{3}-\frac{(\mu \Delta t)^3}{9})
(1-\mu \Delta t+\frac{2(\mu \Delta t)^2}{9})
(E^2-1)W_{2n}|\\
&+|\sigma \frac{8\mu \Delta t}{3}
(1-\frac{\mu \Delta t}{3})
(1+\frac{2\mu \Delta t}{3}-\frac{(\mu \Delta t)^3}{9})
\Delta W_{2n}|
   \\
 &+
 |\sigma^2 
\frac{4\mu \Delta t}{3}
(1-\frac{\mu \Delta t}{3})(1-\mu \Delta t+\frac{2(\mu \Delta t)^2}{9})
(E^2-1)W_{2n}\Delta W_{2n}|
\bigg)
 \;|\tilde{X}_{2n}|^2,
 \end{split}   
\end{align}
Applying $\mathbb{E}|.|^2$ to \eqref{eq:stability} and using the independence of the Wiener increments, we get
\begin{align}
\begin{split}
\label{eq:eq1}
 \mathbb{E}|\tilde{X}_{2n+2}|^2 &= \bigg(\frac{
 |1+\frac{2\mu \Delta t}{3}-\frac{(\mu \Delta t)^3}{9}|^2+
 \Delta t|\sigma (
 1-\mu \Delta t+\frac{2(\mu \Delta t)^2}{9}
)|^2
+
\Delta t|\sigma \frac{4\mu \Delta t}{3}
(1-\frac{\mu \Delta t}{3})|^2}
{|1-
 \mu \Delta t
 +\frac{(\mu \Delta t)^2}{3}|^2\;|1-\frac{\mu \Delta t}{3}|^2}
\\
&+
\frac{\Delta t|\sigma^2 
\frac{4\mu \Delta t}{3}
(1-\frac{\mu \Delta t}{3})(1-\mu \Delta t+\frac{2(\mu \Delta t)^2}{9})
|}
{|1-
 \mu \Delta t
 +\frac{(\mu \Delta t)^2}{3}|^2\;|1-\frac{\mu \Delta t}{3}|^2}\bigg)
 \mathbb{E}|\tilde{X}_{2n}|^2,
 \end{split}   
\end{align}
Using \eqref{eq:eq1} and \eqref{eq:method2}, one has the following result
 \begin{align}
 \begin{split}
   \lim_{n \to \infty} &\mathbb{E} (|\tilde{X}_{2n+2}|^2)=0 \iff \\
   &\frac{
 |1+\frac{2\mu \Delta t}{3}-\frac{(\mu \Delta t)^3}{9}|^2+
 \Delta t|\sigma (
 1-\mu \Delta t+\frac{2(\mu \Delta t)^2}{9}
)|^2
+
\Delta t|\sigma \frac{4\mu \Delta t}{3}
(1-\frac{\mu \Delta t}{3})|^2}
{|1-
 \mu \Delta t
 +\frac{(\mu \Delta t)^2}{3}|^2\;|1-\frac{\mu \Delta t}{3}|^2} \\&
 +
\frac{\Delta t|\sigma^2 
\frac{4\mu \Delta t}{3}
(1-\frac{\mu \Delta t}{3})(1-\mu \Delta t+\frac{2(\mu \Delta t)^2}{9})
|}
{|1-
 \mu \Delta t
 +\frac{(\mu \Delta t)^2}{3}|^2\;|1-\frac{\mu \Delta t}{3}|^2}<1.
 \end{split}
\end{align}
From \eqref{eq:method1}, we have the same result in a similar manner.
\begin{align}
 \begin{split}
   \lim_{n \to \infty} &\mathbb{E} (|\tilde{X}_{2n+1}|^2)=0 \iff \\
   &\frac{
 |1+\frac{2\mu \Delta t}{3}-\frac{(\mu \Delta t)^3}{9}|^2+
 \Delta t|\sigma (
 1-\mu \Delta t+\frac{2(\mu \Delta t)^2}{9}
)|^2
+
\Delta t|\sigma \frac{4\mu \Delta t}{3}
(1-\frac{\mu \Delta t}{3})|^2}
{|1-
 \mu \Delta t
 +\frac{(\mu \Delta t)^2}{3}|^2\;|1-\frac{\mu \Delta t}{3}|^2} \\
 &+
\frac{\Delta t|\sigma^2 
\frac{4\mu \Delta t}{3}
(1-\frac{\mu \Delta t}{3})(1-\mu \Delta t+\frac{2(\mu \Delta t)^2}{9})
|}
{|1-
 \mu \Delta t
 +\frac{(\mu \Delta t)^2}{3}|^2\;|1-\frac{\mu \Delta t}{3}|^2}<1.
 \end{split}
\end{align}
and hence in general we get
 \begin{align}
 \begin{split}
   \lim_{n \to \infty} &\mathbb{E} (|\tilde{X}_{n}|^2)=0 \iff \\
   &\frac{
 |1+\frac{2\mu \Delta t}{3}-\frac{(\mu \Delta t)^3}{9}|^2+
 \Delta t|\sigma (
 1-\mu \Delta t+\frac{2(\mu \Delta t)^2}{9}
)|^2
+
\Delta t|\sigma \frac{4\mu \Delta t}{3}
(1-\frac{\mu \Delta t}{3})|^2}
{|1-
 \mu \Delta t
 +\frac{(\mu \Delta t)^2}{3}|^2\;|1-\frac{\mu \Delta t}{3}|^2} \\&
 +
\frac{\Delta t|\sigma^2 
\frac{4\mu \Delta t}{3}
(1-\frac{\mu \Delta t}{3})(1-\mu \Delta t+\frac{2(\mu \Delta t)^2}{9})
|}
{|1-
 \mu \Delta t
 +\frac{(\mu \Delta t)^2}{3}|^2\;|1-\frac{\mu \Delta t}{3}|^2}<1.
 \end{split}
\end{align}
We conclude that under the following condition stability holds:
\begin{align}
\label{eq:stability_condition1}
\begin{split}
 &\frac{
 |1+\frac{2\mu \Delta t}{3}-\frac{(\mu \Delta t)^3}{9}|^2+
 \Delta t|\sigma (
 1-\mu \Delta t+\frac{2(\mu \Delta t)^2}{9}
)|^2
+
\Delta t|\sigma \frac{4\mu \Delta t}{3}
(1-\frac{\mu \Delta t}{3})|^2}
{|1-
 \mu \Delta t
 +\frac{(\mu \Delta t)^2}{3}|^2\;|1-\frac{\mu \Delta t}{3}|^2}\\&
 +
\frac{\Delta t|\sigma^2 
\frac{4\mu \Delta t}{3}
(1-\frac{\mu \Delta t}{3})(1-\mu \Delta t+\frac{2(\mu \Delta t)^2}{9})
|}
{|1-
 \mu \Delta t
 +\frac{(\mu \Delta t)^2}{3}|^2\;|1-\frac{\mu \Delta t}{3}|^2}<1. 
 \end{split}
\end{align}

\end{proof}
\begin{definition}
  The local error of $\{\tilde{X}(t_n)\}$   between two consecutive time is the sequence of random variables is: $\delta_n=X(t_n)-\tilde{X}(t_n)$.
  The local error measures the difference between the approximation and the exact
solution on a subinterval of the integration.
\end{definition}
\begin{definition}
The sequence $\{\tilde{X}_{n}\}_{n=0}^{\infty}$ is mean-square consistent if
\begin{equation}
  \mathbb{E} (|\delta_n|^2)\leq C \Delta t ^p\quad \text{as} \quad \Delta t \to 0.
\end{equation}
for some positive constants $p\geq 1$ and $C$.
\end{definition}

\begin{theorem}
\label{theorem:consistency}
the proposed method
\eqref{eq:method}
is consistent 
in mean square 
for approximating the solution of   stochastic diffusion equation \eqref{eq:sde}.
\end{theorem}

\begin{proof}
From \eqref{eq:sde}, the exact solution $X(t_{2n+2})$ can be expanded using the Itô formula as follows
\begin{equation}
 X(t_{2n+2})= X(t_{2n})+ \mu \int_{t_{2n}}^{t_{2n+2}} X(s)\; ds+ \sigma \int_{t_{2n}}^{t_{2n+2}} X(s)\; dW(s).   
\end{equation}
$\tilde{X}(t_{2n+2})$ denotes the locally approximate value obtained after just one step of equation \eqref{eq:m1}, i.e.
\begin{equation}
\tilde{X}(t_{2n+2})= X(t_{2n})+ \mu \int_{t_{2n}}^{t_{2n+2}} (\sum_{j=2n}^{2n+2} \psi_{n,j}(t)X(t_j))\; dt+ \sigma X(t_{2n})
(E^2-1)W_{2n}.
\end{equation}
Therefore, we get
\begin{align}
\begin{split}
 \mathbb{E} &(|\delta_{2n+2}|^2)\\
 &=    \mathbb{E} (|
 \mu \int_{t_{2n}}^{t_{2n+2}} (X(s)-
 \sum_{j=2n}^{2n+2} \psi_{n,j}(t)X(t_j))\; ds+ \sigma \int_{t_{2n}}^{t_{2n+2}} (X(s)-X(t_{2n}))\; dW(s)
 |^2)\\
 &\leq 
 2\mu ^2 \mathbb{E} |\int_{t_{2n}}^{t_{2n+2}} (X(s)-
 \sum_{j=2n}^{2n+2} \psi_{n,j}(t)X(t_j))\; ds
 |^2
 +2\sigma^2 \mathbb{E} | \int_{t_{2n}}^{t_{2n+2}} (X(s)-X(t_{2n}))\; dW(s)|^2,
 \end{split}
\end{align}
Using the interpolation error formula, we get
\begin{align}
\begin{split}
 \mathbb{E} (|\delta_{2n+2}|^2)&\leq 
 2\mu ^2 \mathbb{E} |\int_{t_{2n}}^{t_{2n+2}} 
 X[t_{2n},t_{2n+1},t_{2n+2},s] (s-t_{2n})
 (s-t_{2n+1})
 (s-t_{2n+2})
 \; ds
 |^2\\
 &+2\sigma^2  \int_{t_{2n}}^{t_{2n+2}} \mathbb{E}|X(s)-X(t_{2n})|^2 ds,
 \end{split}
\end{align}
Note that
\begin{equation}
   X(s) -X(t_{2n})\approx \mu X(t_{2n} )(s-t_{2n})
   +\sigma X(t_{2n}) (W_s-W_{2n}).
\end{equation}
Therefore, when the time step $\Delta t$ tends to zero, there exists $k$ such that
\begin{equation}
 \mathbb{E} (|\delta_{{2n+2}}|^2)   \leq k \Delta t. 
\end{equation}
Similarly, we have the same result for \eqref{eq:res3}, i.e. there exists $k'$ such that
\begin{equation}
 \mathbb{E} (|\delta_{{2n+1}}|^2)   \leq k' \Delta t. 
\end{equation}
This proves the consistency. So, the proposed scheme is consistent in mean square.
\end{proof}

\begin{remark}
According to the theorems proved about the stability and consistency of the proposed scheme
and the stochastic version of the Lax- Richtmyer theorem \cite{Roth2001}, our method is conditionally convergent for solving stochastic differential equation \eqref{eq:sde}. 
\end{remark}

\section{Numerical results}
\label{sec:numerical}
In this section, the performance of the presented numerical technique described in section \ref{sec:methodology} for solving 
\eqref{eq:sde}
are considered and applied to some test problems. 

The errors reported are using the
$L^1$-error, $L^2$-error and $L^{\infty}$-error between the approximation and exact solutions:
\begin{align*}
\begin{split}
  \Vert e \Vert _{L^1} &= \frac{1}{N}\sum_{i=0}^N|X(t_i)-\tilde{X}_i |,  
\\
  \Vert e \Vert _{L^2} &= \sqrt{\frac{1}{N}\sum_{i=0}^N|X(t_i)-\tilde{X}_i |^2},   
\\
  \Vert e \Vert _{L^{\infty}} &= \max_{0\leq i\leq N}|X(t_i)-\tilde{X}_i |.  
  \end{split}
\end{align*}
where $X(t_i)$ and $\tilde{X}_i$ are the exact and approximate solutions, respectively.

We examine the performance of the proposed scheme for stochastic differential equation of the form:
\begin{equation}
\label{eq:test}
  dX(t)=\mu X(t) \;dt+\sigma X(t) \;dW(t), ~0\leq t \leq 1,  
\end{equation}
subject to the following initial condition
\begin{equation}
\label{eq:initial}
  X(0)=1.  
\end{equation}
The exact solution of \eqref{eq:test} with initial condition \eqref{eq:initial} is given by
\begin{equation}
  X(t)=X(0)e^{(\mu-\frac{\sigma^2}{2})t+\sigma W(t)}.  
\end{equation}
We set   $\sigma=0.5$.
We have plotted in Figure 
\ref{fig:stability} the stability region for our method when $\sigma=0.5$. 
Figure \ref{fig:stability_others} shows the stability region for implicit EM and Milstein methods.
Note that we have almost the same stability region for all the strategies. 
We assume $\mu=-1$.
Figure \ref{fig:exact_approx}
shows the stochastic solution
using  our scheme 
along with the exact  solution for
$\Delta t=2^{-8}$. 
They put on show the successful behaviour of our approach for solving  \eqref{eq:sde}.
The $L^1$-error, $L^2$-error and $L^{\infty}$-error are shown in Table \ref{tab:res_l1}, Table \ref{tab:res_l2} and Table \ref{tab:res_linf}, respectively.
Firstly, we show that in our method, the error decreases consistently under mesh refinement.
Under these assumptions, the stability
condition is satisfied.
Our method has
a superior performance when compared to
implicit EM and Milstein methods.

The numerical solution obtained with $ \mu=1$
is shown in Figure \ref{fig:instability}. 
We observe sample paths that deviate wildly, illustrating instability.
We observe that the computational results justify the theory of the stability conditions.

\begin{figure}[H]
\begin{center}
{\includegraphics[width=0.45\textwidth]{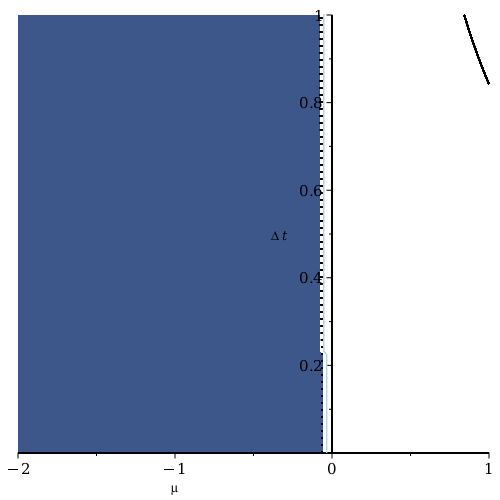}}
	\caption{Plot of the stability region (shaded area) determined by \eqref{eq:stability_condition1} for $\sigma=0.5$. 
}
	\label{fig:stability}
\end{center}
\end{figure}

\begin{figure}[H]
\begin{center}
\subfigure[Implicit EM]
{\includegraphics[width=0.45\textwidth]{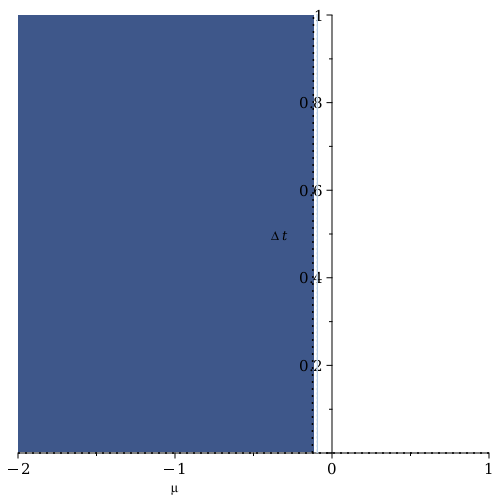}
 \label{fig:stability_euler}}
 \subfigure[Milstein]
{\includegraphics[width=0.45\textwidth]{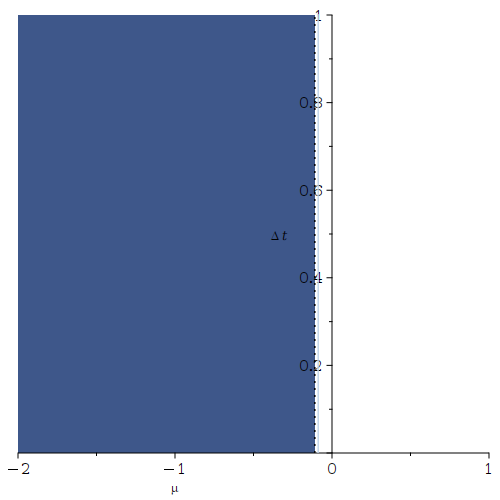}
 \label{fig:stability_milstein}}
\end{center}
	\caption{Plot of the stability region (shaded area)  for $\sigma=0.5$. Implicit EM on the left and Milstein on the right.}
 \label{fig:stability_others}
\end{figure}

\begin{figure}[H]
\begin{center}
{\includegraphics[width=0.45\textwidth]{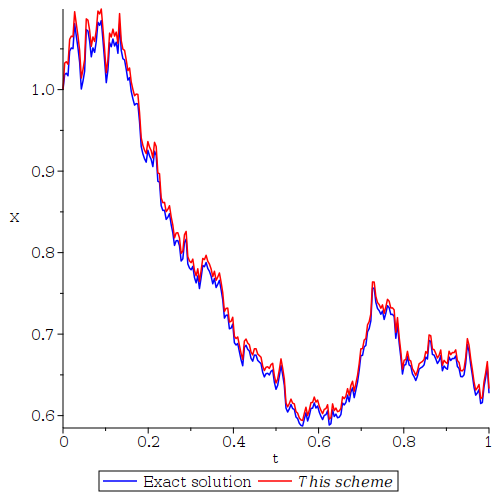}}
	\caption{Approximate and analytical solution of equation \eqref{eq:test} using the proposed method. 
}
	\label{fig:exact_approx}
\end{center}
\end{figure}

\begin{table}[H]
  \centering
  \caption{The $\Vert e \Vert_{L ^1}$ error using different methods for equation \eqref{eq:test}.
 } 
\begin{tabular}{|c ||c c c|} 
\hline
$N$ & This scheme & Implicit EM & Milstein\\
\hline
$2^2$ & $\textbf{5.1669e-03}$ & $2.2051e-02$  & $4.7405e-02$\\
$2^4$ & $\textbf{2.3844e-03}$ & $6.3626e-03$  & $1.2747e-02$\\
$2^6$ & $\textbf{6.9950e-04}$ & $1.6750e-03$  & $3.2850e-03$\\
$2^8$ & $\textbf{1.9056e-04}$ & $4.2290e-04$  & $8.3239e-04$\\
$2^{10}$ & $\textbf{5.0966e-05}$ & $1.0671e-04$  & $2.0978e-04$\\
\hline
\end{tabular}
	\label{tab:res_l1}
\end{table}

\begin{table}[H]
  \centering
  \caption{The $\Vert e \Vert_{L ^2}$ error using different methods for equation \eqref{eq:test}.
 } 
\begin{tabular}{|c ||c c c|} 
\hline
$N$ & This scheme & Implicit EM & Milstein\\
\hline
$2^2$ & $\textbf{7.6716e-03}$ & $2.7473e-02$  & $5.3226e-02$\\
$2^4$ & $\textbf{3.1373e-03}$ & $7.3833e-03$  & $1.3261e-02$\\
$2^6$ & $\textbf{8.9318e-04}$ & $1.8964e-03$  & $3.3430e-03$\\
$2^8$ & $\textbf{2.3630e-04}$ & $4.7560e-04$  & $8.4199e-04$\\
$2^{10}$ & $\textbf{6.3238e-05}$ & $1.2002e-04$  & $2.1162e-04$\\
\hline
\end{tabular}
	\label{tab:res_l2}
\end{table}

\begin{table}[H]
  \centering
  \caption{The $\Vert e \Vert_{L ^{\infty}}$ error using different methods for equation \eqref{eq:test}.
 } 
\begin{tabular}{|c ||c c c|} 
\hline
$N$ & This scheme & Implicit EM & Milstein\\
\hline
$2^2$ & $\textbf{1.5005e-02}$ & $4.1655e-02$  & $6.4008e-02$\\
$2^4$ & $\textbf{6.5865e-03}$ & $1.1182e-02$  & $1.5031e-02$\\
$2^6$ & $\textbf{1.9095e-03}$ & $2.8493e-03$  & $3.7154e-03$\\
$2^8$ & $\textbf{5.0434e-04}$ & $7.1227e-04$  & $9.2942e-04$\\
$2^{10}$ & $\textbf{1.3240e-04}$ & $1.8051e-04$  & $2.3177e-04$\\
\hline
\end{tabular}
	\label{tab:res_linf}
\end{table}

\begin{figure}[H]
\begin{center}
{\includegraphics[width=0.45\textwidth]{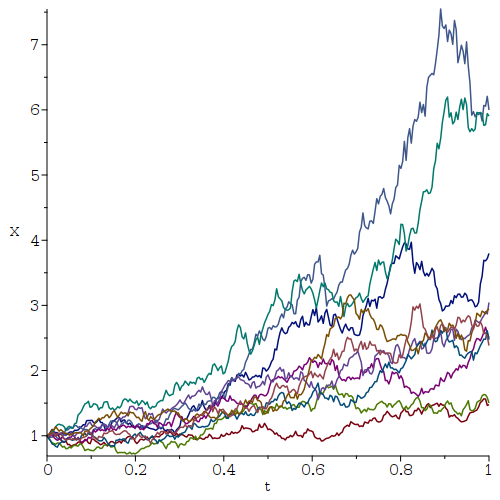}}
	\caption{10 simulations using  the proposed method for $\mu=1$.
}
	\label{fig:instability}
\end{center}
\end{figure}

\section{Conclusion}
\label{sec:conclusion}

We presented a methodology for the numerical solution of stochastic differential equations.
We subdivide the time interval  into sub-intervals.
We approximate the
drift term over
three successive node
points using quadratic polynomials.
 The most
important properties of a stochastic numerical
schemes, namely,  stability, consistency and convergence  have been described and analyzed. 
The proposed method have been illustrated by numerical examples. The numerical experiments demonstrate
that using our method leads
to a better performance in comparison to implicit EM and Milstein schemes.
\\
In the future, we would like to extend such methods for  the general form of the SDEs.

\bibliographystyle{unsrt}
\bibliography{references} 

\end{document}